\documentclass[reqno]{amsart}

\usepackage{graphicx,pdfsync,amssymb,mathrsfs,amsfonts,amsbsy,color,esint,mathtools,tikz,mdframed}
\usetikzlibrary{positioning}
\usepackage[shortlabels]{enumitem}
\usepackage{makecell}

\mathtoolsset{showonlyrefs}

\usepackage{booktabs}
\usepackage{array}
\usepackage[table]{xcolor}
\usepackage{hyperref}
\usepackage{amsfonts} 

\hypersetup{
colorlinks=true,
linkcolor=blue,
filecolor=blue,
urlcolor=blue,
citecolor=blue,
}

\usepackage[utf8]{inputenc}
\usepackage[T1]{fontenc}

\DeclareMathOperator{\Supp }{supp}

\DeclareMathOperator{\sign}{sign}

\newtheorem{theorem}{Theorem}[section]
\newtheorem{lemma}[theorem]{Lemma}
\newtheorem{proposition}[theorem]{Proposition}

\newtheorem{corollary}[theorem]{Corollary}
\newtheorem{remark}[theorem]{Remark}

\setlist[enumerate]{itemsep=3pt}
\setlist[itemize]{itemsep=3pt}

\def \T  {\mathbb{T}} 
\def \R {\mathbb{R}}  

\def \p {\partial}

\def \ep {\epsilon}
\def \om {\omega}
\def \Om {\Omega}

\numberwithin{equation}{section}

\begin{document}


\title[Asymptotic stability of shear flows]{Asymptotic stability of shear flows for 2D Euler equations at Yudovich regularity}

\author{Dengjun Guo}

\address{Academy of Mathematics and Systems Science, Chinese Academy of Sciences, Beijing, China}

\email{djguo@amss.ac.cn}

\author{Xiaoyutao Luo}

\address{Academy of Mathematics and Systems Science, Chinese Academy of Sciences, Beijing, China}

\email{xiaoyutao.luo@amss.ac.cn}

 
\date{\today}

\begin{abstract}
The nonlinear asymptotic stability of shear flows in the 2D Euler equations has traditionally been linked to inviscid damping in the periodic setting. Since Gevrey regularity is required to suppress the ``echo'' phenomenon, asymptotic stability is known to be impossible in Sobolev spaces.

In this paper, we identify a distinct stabilizing mechanism available in the infinite channel: the advection of vorticity to spatial infinity. We establish nonlinear asymptotic stability for the 2D Euler equations in the infinite channel $\mathbb{R}\times[0,1]$ at the minimal regularity of the Yudovich class ($L^{\infty}$ vorticity). Specifically, for a class of non-negative shear flows with a curvature bound, any $L^\infty$-small, compactly supported vorticity perturbation leads to decay on compact subsets and weak convergence to zero.  
\end{abstract}

\maketitle

\section{Introduction}\label{sec:intro}

Shear flows constitute a fundamental class of coherent, steady solutions to the 2D Euler equations of an incompressible, ideal fluid. Their stability has been a subject of intense study since the classical linear analyses of Kelvin~\cite{Kelvin1887} and Rayleigh~\cite{Rayleigh1879} in the late nineteenth century. Even though the foundational mechanisms of linear stability and algebraic decay were identified early on, establishing nonlinear asymptotic stability in the 2D Euler equations remained elusive for over a century. While Arnold's variational methods~\cite{MR1612569} successfully established nonlinear \emph{orbital stability} for certain steady flows, such frameworks are inherently unequipped to capture the long-time asymptotic convergence to a steady state.

The first nonlinear asymptotic stability in the 2D Euler equations was achieved much more recently in the landmark breakthrough by Bedrossian and Masmoudi~\cite{BedrossianMasmoudi15}. Following this work, existing nonlinear stability results for shear  flows~\cite{BedrossianMasmoudi15,MR4076093,MR4628607,MR4740211} are restricted to the periodic setting. Relying on the delicate mechanism of inviscid damping, the perturbation is required to be Gevrey regular~\cite{MR2796139,MR4630602}. This regularity requirement is sharp: below a certain Gevrey class, long time instability occurs~\cite{MR4630602}, and in Sobolev spaces, nontrivial steady states can persist arbitrarily close to the shear flow~\cite{MR2796139,MR4595614}, precluding asymptotic relaxation in the periodic setting.

The stringent regularity requirement stands in stark contrast with the physical reality of ideal fluids. While classical works by Wolibner and H\"older in the 1930s established global well-posedness for smooth data, the natural state space for the 2D Euler equations is much rougher. In 1963, Yudovich~\cite{MR158189} made a fundamental breakthrough by proving global existence and uniqueness for initial data with merely bounded vorticity ($L^1 \cap L^{\infty}$), completely dispensing with Sobolev or H\"older regularity constraints. This functional setting, now known as the Yudovich class, mirrors the underlying Eulerian dynamics, where vorticity is transported along particle trajectories and has been proven to be sharp~\cite{Vishik1,Vishik2,ABCDGJK}.

Yet, over six decades since the Yudovich theory of global wellposedness, extending \emph{asymptotic stability} to such rough data has remained a major challenge. This difficulty reflects a broader phenomenon in the analysis of nonlinear PDEs: it is exceptionally rare for a nonlinear system to exhibit asymptotic stability exactly at the threshold of its well-posedness, particularly when lacking dissipation. Because the Euler equations lack any smoothing mechanism, rough perturbations remain rough forever; in fact, one of the hallmarks of 2D Euler equations is the ``generic'' loss of smoothness~\cite{MorgulisShnirelmanYudovich08,DrivasElgindiJeong24,alazardsaid2026} and small scale formation~\cite{Denisov09,KS14,Z24} in infinite time (see Section \ref{subsub:growth} for a discussion of underlying growth mechanisms).

This massive ``regularity gap'' between well-posedness theory and stability theory raises a fundamental question:
\vspace{0.5em}
\begin{center}
\textit{Are there any steady states of the 2D Euler equations that are asymptotically stable at the Yudovich regularity?}
\end{center}
\vspace{0.5em}

In this paper, we answer this question affirmatively for a large family of shear flows in the infinite channel  $\R \times [0,1]$. We propose a distinct stability mechanism available in the infinite channel: the advection of vorticity to spatial infinity. We show that this mechanism is robust enough to stabilize perturbations of Yudovich class---a minimal regularity for which the 2D Euler equations are well-posed. This represents the first asymptotic stability result at the Yudovich regularity for the 2D Euler equations.

To formalize the setting, consider the 2D Euler equations in the infinite channel $\Omega: = \R \times [0,1]$. Given a $C^2([0,1])$ shear flow profile $f(y)$,  the vorticity perturbation $\om : \Omega \times [0,\infty) \to \R$ satisfies 
\begin{equation}\label{eq:eu_eq}
\begin{cases}
	\p_t \om + f(y)\p_x \om + u \cdot \nabla \om = -f''(y)u^y &\\ 
	\om|_{t= 0 } = \om_{in} &
\end{cases}	
\end{equation}
where the velocity $u: \Om \times [0,\infty)  \to \R^2$ is recovered via $u=(u^x,u^y) = \nabla^{\perp}\Delta^{-1}\omega$,  $ \nabla^\perp = (- \p_y , \p_x)$,  subject to no-penetration boundary conditions.

\subsection{Main result}
The main result below establishes asymptotic stability for general $C^2$ shear flows subject to Yudovich perturbations. This relaxes the regularity assumptions on both the background flow and the perturbation well beyond the Gevrey classes required in previous studies.
 
We begin by defining the class of admissible shear flow profiles $\mathcal{F} \subset C^2([0,1])$. Let $C_*>0$ be a universal constant to be specified, and denote by $d(y) := \min \{y, 1-y \}$ the distance to the channel boundary  $\p \Omega$. We define the class $f \in \mathcal{F}$ such that there exists $\delta > 0$ satisfying:
\begin{itemize}
    \item Non-negativity:
    \begin{equation}\tag{H1}\label{eq:thm:shear assump f 1}
    f(y) \ge \delta d(y) .
    \end{equation}

    \item Curvature bound:
    \begin{equation}\tag{H2}\label{eq:thm:shear assump f 2}|f''(y)| \leq C_* \delta d(y).
    \end{equation}
\end{itemize}
\begin{theorem}\label{thm:shear}
There exists a universal constant $C_*>0$ such that for any shear flow $f \in \mathcal{F}$, \eqref{eq:eu_eq} is asymptotically stable in the infinite channel $\Omega = \mathbb{R} \times [0,1]$ with respect to compactly supported Yudovich perturbations.

More precisely, for any $f \in \mathcal{F}$, there exists $\epsilon = \epsilon(f) > 0$ such that for any compactly supported $\omega_{in} \in L^\infty(\Omega)$ satisfying
\begin{equation}\label{eq:thm:shear perturb assump}
\|\omega_{in}\|_{L^\infty(\Omega)} \leq \ep,
\end{equation}
the unique Yudovich solution\footnote{Despite the unboundedness of $\Omega$ and the non-decaying background shear, the solution preserves the $L^1 \cap L^\infty$ regularity (but not its size) of the initial data; see Section \ref{sec:shear}} $\omega(t)$ to \eqref{eq:eu_eq} satisfies
\begin{equation}\label{eq:thm:shear conclusion}
\lim_{t \to \infty} \|\omega(t)\|_{L^1 (K)} = 0 \quad \text{for any compact set $K \subset \Om$. }
\end{equation}

Consequently, $\omega(t) \rightharpoonup 0$ weakly in $L^p(\Omega)$ for every $1 < p < \infty$ as $t\to \infty $.
\end{theorem}

\begin{remark}[Stability mechanism]

In the infinite channel $\mathbb{R}\times[0,1]$, asymptotic stability arises from the advection of vorticity to spatial infinity---a macroscopic phenomenon we term \textbf{vorticity escape}---essentially the formation of traveling wave-like structures. 

This mechanism should not be confused with passive transport by a given shear $f(y) \geq 0$. For the passive equation
\[
\partial_t\omega+f(y)\partial_x\omega=0,
\]
the vorticity is transported with no back-flow region. In the nonlinear Euler dynamics, by contrast, the perturbation may produce regions with adverse horizontal velocity that can in principle trap vorticity, especially near the boundary where the background shear can vanish.

For the Couette flow $f(y) = y \in \mathcal{F}$, the low $L^\infty$ regularity assumption contrasts with the periodic case~\cite{BedrossianMasmoudi15,MR4076093}, where the confinement of vorticity necessitates high Gevrey regularity.

\end{remark}

\begin{remark}[Necessity of no stagnation]
With translational structures stabilizing the flow, the only remaining obstruction is \textbf{stationarity}.

While the condition \eqref{eq:thm:shear assump f 1} allows for boundary degeneracy ($f(0)=0$ or $f(1) =0$), it implies $ f > 0$ in the interior, preventing the formation of  interior trapped vortices. This structural condition is sharp: in joint work~\cite{steady} with G. Qin, we prove the existence of steady states near such a sign-changing shear flow.

\end{remark}

\begin{remark}[The applicability]
Our class $\mathcal{F}$ includes all linear Couette flows $(ay,0)  $ with $a>0$ and nearby non-negative $C^2$ perturbations satisfying \eqref{eq:thm:shear assump f 2}. However, the smallness of $C_*$ excludes   flows that vanish at both walls.

 For strictly positive shears $\inf f>0$, the criterion reduces to a curvature upper bound proportional to $\inf f$. The delicate case in Theorem~\ref{thm:shear}  is therefore the boundary-degenerate regime, where the self-induced back-flow discussed above must be controlled without a uniform lower bound on the background velocity.


\end{remark}

\subsection{Exponential decay for non-stagnant flow}
Theorem \ref{thm:shear} provides a robust criterion for shear flow asymptotic stability without a convergence rate. The degeneracy near the boundary impedes uniform decay rates. We show that for non-stagnant flows, namely $f(y) \ge c >0$,  the stabilizing mechanism of advection yields exponential decay of the local vorticity, manifesting the exponential decay of the Dirichlet Green function in the infinite channel.

\begin{theorem}\label{thm:shear2 f>c}
There exists a universal constant  $C_*>0$ such that the following holds. Let $f \in C^2([0,1])$ be a shear profile satisfying 
\begin{equation}\label{eq:thm 2 f >c shear assump}
m_f=   \inf_{y\in [0,1]} f(y) > 0, \qquad      \| f''\|_{L^\infty([0,1])}  \le C_*  m_f . 
\end{equation}
Then  for any compactly supported $ \om_{ in} \in  L^\infty(\Omega)$  satisfying
\begin{equation}\label{eq:thm 2 f >c perturb assump}
\|\om_{ in}  \|_{L^\infty(\Omega)} \leq  C_*  m_f  , 
\end{equation}  
the unique Yudovich solution  $\om(t)  $ to \eqref{eq:eu_eq} satisfies, for any $R>0$,  
 \begin{equation}\label{eq:thm:shear2 conclusion}
        \|\om(t) \|_{ L^\infty ( (-\infty,R]\times [0,1])} \lesssim_{R} e^{- \frac{1}{2} m_f t} \qquad   t \ge 0  .
    \end{equation}

\end{theorem}

 \begin{remark}

The condition $\inf f > 0$ is exactly the rigidity criterion appearing in recent stationary classification results~\cite{HamelNadir17,HamelNadir19}. This even allows shear flows to be non-monotone.

\end{remark}

\begin{remark}
We note that in both Theorem  \ref{thm:shear} and Theorem \ref{thm:shear2 f>c}:

\begin{itemize}
    \item As there is no restriction on its support size, the perturbation can be arbitrarily large in all Sobolev spaces $W^{k,p}$ for $p<\infty$.

     \item The compact support assumption is necessary to rule out nearby periodic~\cite{MR2796139,MR4595614} or quasi-periodic~\cite{MR4796775} stationary structures\footnote{Without the compact support assumption, the  simplest counterexample is a nearby shear flow.} that do not decay. 

    \item The initial perturbation is allowed to touch the boundary $y=0$, and it will remain in contact for all times, cf.~\cite{MR4076093}. 
\end{itemize}

Using the last point above, one can devise linear-in-time filamentation for a large class of initial data and affine shears $(by+c,0)$: there are open sets of initial data such that for any $\alpha>0$ there holds
\begin{equation}
\|\om (t) \|_{C^\alpha} \gtrsim t^{\alpha} \quad \text{for all $t\geq 0$}.
\end{equation}

\end{remark}

\subsection{Loss of compactness}
Our weak convergence results relate to \v{S}ver\'ak's generic non-compactness conjecture~\cite{Snotes} and the earlier formulation of ``wandering'' behavior proposed by Nadirashvili~\cite{Nadir91}. While establishing this conjecture in full generality remains open, non-compactness has been demonstrated in the perturbative regime~\cite{BedrossianMasmoudi15} as a consequence of phase mixing.

While the conjecture was originally intended for the bounded domain, our asymptotic stability in the infinite channel also manifests a loss of compactness. However, in the present setting, the mechanism differs significantly. Unlike the periodic case, where weak convergence stems from infinite-time phase mixing, here it arises from the macroscopic transport of bulk vorticity to spatial infinity. This vorticity escape implies the non-compactness of the orbit $t \mapsto \omega(t)$ for the weakly convergent Yudovich solutions established in Theorem \ref{thm:shear} and Theorem \ref{thm:shear2 f>c}.

\begin{theorem}\label{thm:noncompact}
Let  $\om : \Omega \times [0,\infty) \to \R$ be a Yudovich solution to \eqref{eq:eu_eq} with some background shear flow $f \in C^2$ and compactly supported initial data. Denote by $u $  its associated velocity field.

If  $\omega(t) \rightharpoonup 0$ weakly in $L^p(\Omega)$ for every $1 < p < \infty$  as $t\to \infty $,  then depending on the shear profile $f(y)$  the following  loss of relative compactness statements hold. 
\begin{enumerate}
    \item For Couette flows $(by+c, 0)$:
    \begin{itemize}
\item The trajectory $t \mapsto \omega(t)$ is not relatively compact in $L^p(\Omega)$ for any $1\le p \le \infty $.

\item If the initial excess energy \eqref{eq:def_energy} is non-zero, then $t \mapsto u(t)$ is not relatively compact in $L^1(\Omega)$.
\end{itemize}

\item For general shear flows:
\begin{itemize}
\item If the initial excess enstrophy \eqref{eq:def_enstrophy} is non-zero, then the trajectory $t \mapsto \omega(t)$ is not relatively compact in $L^1(\Omega)$.

\item If the initial excess energy \eqref{eq:def_energy} is non-zero, then neither $t \mapsto \omega(t)$ nor $t \mapsto u(t)$ is relatively compact in $L^1(\Omega)$.
\end{itemize}

\end{enumerate}

\end{theorem}

\begin{remark} \hfill 
\begin{itemize} 
\item Although phenomena of vorticity escape are characteristic of  Euler flows in unbounded domains~\cite{MR3117517,AbeChoi22,MR4350517}, this work appears to be the first to leverage this mechanism to establish weak convergence to zero (asymptotic stability).

\item The loss of the excess kinetic energy in infinite time also manifests as the irreversibility of the 2D Euler dynamics. This behavior stands in sharp contrast to inviscid damping in the periodic setting, where the velocity converges strongly in $L^p$. 

\item The assumption of non-zero excess energy or enstrophy represents the ``generic'' case. The latter condition holds automatically for all monotone shears and can be further relaxed to non-zero excess Casimirs.

\end{itemize}
\end{remark}

Theorem \ref{thm:noncompact} follows directly from weak convergence in the presence of conserved quantities---specifically, the excess energy and excess enstrophy. These are the natural analogues of the standard conserved quantities adapted to the infinite channel setting with a non-decaying background shear. See Section \ref{subsec:compactness} for details.

\subsection{Comparison of stability results}

Before coming to the proof, we briefly discuss our result in the context of hydrodynamic stability and long-time dynamics for the 2D Euler equations. As the stability of ideal fluids is a vast subject, we will refer to the surveys \cite{DrivasElgindi23,KMS23} for more related references.

\begin{table}[h!]
\centering
\small
\renewcommand{\arraystretch}{1.6}
\setlength{\tabcolsep}{5pt}
\begin{tabular}{@{} l l l l @{}}
\toprule
\textbf{Setting} & \textbf{Mechanism} & \textbf{Perturbation} & \textbf{Reference} \\
\midrule

\textbf{Gevrey Shear \& $\T \times \R$ or $\T \times [0,1]$} & Phase Mixing &   Gevrey &  \cite{BedrossianMasmoudi15,MR4628607,MR4740211} \\[0.5em]

\textbf{Point Vortex \& $\R^2$} & Phase Mixing &   Gevrey &  \cite{MR4400903}   \\[0.5em]

\textbf{0-homogeneous \& m-fold} &  Spiraling & Not applicable &   \cite{ElgindiMurraySaid25}  \\[0.5em]

\midrule
\rowcolor{blue!10} 
\textbf{$C^2$ Shear \& $   \R\times [0,1]$} & \textbf{Advection} &    $L^\infty$ & This paper   \\

\bottomrule
\end{tabular}
\caption{The landscape of nonlinear stability/relaxation for the 2D Euler equations.}
\label{tab:euler_map}
\end{table}

\subsubsection{The asymmetry and challenge}\label{subsub:growth}
Establishing asymptotic stability for the 2D Euler equations has proven significantly more challenging than proving instability. This difficulty stems from a fundamental physical asymmetry: while ideal fluids lack explicit stabilizing mechanisms like viscous dissipation, they possess robust mechanisms for growth, such as linear stretching \cite{Koch02}, hyperbolic saddle~\cite{MR1288809}, and boundary-driven growth~\cite{Nadir91,KS14}. A significant manifestation of this is the phenomenon of norm inflation in borderline or supercritical spaces ~\cite{BourgainLi14,MR3625192,MR4776418} since the pioneering work of Bourgain and Li \cite{BourgainLi13}. For the long-time dynamics, the robust growth mechanisms result in flourishing development in the literature regarding infinite-time instability and loss of regularity \cite{Koch02,MorgulisShnirelmanYudovich08,Denisov09,KS14,DrivasElgindiJeong24,Z24,MR4350517,jeongyaozhou2025,alazardsaid2026}.

In stark contrast, even in the linearized setting, proving infinite-time convergence remains highly non-trivial \cite{MR3772399,MR3987441}. Table \ref{tab:euler_map} summarizes this sparse landscape of rigorous nonlinear convergence results. Because the equations lack explicit dissipation to damp out perturbations, any asymptotic stabilization must emerge purely from the kinematics of the nonlinear flow itself. For over a century~\cite{Kelvin1887,Rayleigh1879,Orr1907}, the sole kinematic mechanism known to achieve this was inviscid damping via phase mixing.

\subsubsection{The phase mixing paradigm}

Prior to this work, the predominant mechanism of establishing nonlinear asymptotic stability for shear flows has been phase mixing. This paradigm was pioneered by Bedrossian and Masmoudi ~\cite{BedrossianMasmoudi15} for the Couette flow in $\T\times \R$ and extended to a class of  monotone shear flows ~\cite{MR4628607,MR4740211} in $\T\times [0,1]$. The phase mixing mechanism relies on not only    the periodic setting, but also Gevrey regularity of the initial perturbation. This high-regularity requirement also extends to the asymptotic stabilization of point vortices in $\mathbb{R}^2$~\cite{MR4400903}, which similarly leverages phase mixing in the angular coordinate.

Beyond phase mixing, rigorously identified relaxation mechanisms remain rare and geometrically constrained. Recently, the authors in \cite{ElgindiMurraySaid25} demonstrated relaxation for a class of scale-invariant solutions in $\R^2$, bypassing the need for a background flow. While this represents the first non-perturbative relaxation/convergence result, the admissible class of initial data is highly specialized: $0$-homogeneous and $m$-fold symmetric with $m \ge 4$.

\subsubsection{Orbital vs. asymptotic stability}
 
Finally, from a variational perspective, we also note the orbital stability results \cite{WanPulvirenti85,MR3117517,AbeChoi22} established in the spirit of Arnold~\cite{MR1612569}. While these frameworks successfully identify stable stationary structures and establish orbital stability in the form of $L^p$ boundedness of the vorticity, such frameworks have not yielded asymptotic convergence to date.

\subsection{Methodology and outline}
 
Our method introduces a distinct stabilizing paradigm: tracking the macroscopic escape of vorticity ``downstream''. We develop two Eulerian approaches below based on this advection mechanism and transport estimates, allowing us to completely bypass the Gevrey regularity barrier and close the stability estimates at the minimal Yudovich ($L^\infty$) regularity.

While this mechanism is conceptually intuitive, the primary analytical challenge is controlling the linear stretching term $-f''(y)u^y$ and the back-flow near the boundary. Even as the background shear transports fluid downstream, the linear stretching can, in principle, generate exponential growth of the perturbation. This destabilizing threat is most severe at the channel boundaries where the background velocity can degenerate.

To overcome these challenges, the remainder of the paper is organized as follows:

\begin{itemize}
    \item Section \ref{sec:shear}: We establish the general framework for Yudovich solutions, notably velocity estimates and propagation of exponential integrability.  As an immediate application, we prove Theorem \ref{thm:shear2 f>c} by deriving exponential decay of the perturbation for non-stagnant shear flows. Using a \textit{weighted vorticity} of the form $
\Phi(x,y,t) = \omega(x,y,t) e^{-a(x - ct)}$, the exponential decay of the local vorticity follows from a maximum principle for $\Phi$, under the   conditions  $f\ge c >0$ and $|f''| \lesssim f$.

\item Section \ref{sec:renorm}: We introduce the renormalization framework and analyze the system's conserved quantities. First, we adapt the DiPerna-Lions renormalization framework to our specific setting, which lays the crucial groundwork for the energy arguments used later in the paper. Second, we leverage this framework—alongside the natural conserved quantities of our system—to rigorously prove the loss of compactness results stated in Theorem \ref{thm:noncompact}.

\item Section \ref{sec:degenerate}: We prove Theorem \ref{thm:shear}, addressing the critical case where the background velocity vanishes at the boundary.  Since the pointwise maximum principle fails when dealing with this degeneracy alongside rough $(L^{\infty})$ perturbations, we employ a \textit{renormalized energy method} using the framework established in Section \ref{sec:renorm}.  We construct a weighted $L^1$ mass functional to  track the bulk mass of the vorticity. The key technical innovation is that the nonlinear estimate for the $L^1$ weighted mass resembles a double interaction energy integral, and thus enjoys favorable symmetry properties---a feature absent in any other $L^p$ energy functionals ($p > 1$). This crucial symmetry, combined with a careful analysis of the Biot-Savart kernel, proves that the background shear is strong enough to overcome the induced back-flow and ``strip'' the boundary layer.

\item Section \ref{sec:discuss}: We discuss further implications of our stability results in the broader context of   hydrodynamic stability and conclude with some open problems.
\end{itemize}

\subsection*{Acknowledgment}

We thank Jiajun Tong for his critical comments on an earlier version of the manuscript. XL is supported by NSFC No. 12421001 and  No. 12288201.

\section{Yudovich solutions on the infinite channel}\label{sec:shear}

In this section, we set up the weak solution framework of Yudovich in the infinite channel $\Om = \R \times [0,1]$. Due to the two parallel boundaries $y=0$ and $y=1$, the Biot-Savart kernel decays exponentially fast at infinity.

As an application of this exponential screening effect, we prove local-in-space exponential decay for non-stagnant shear flows in Lemma \ref{lemma:max_principle} below, thereby concluding Theorem \ref{thm:shear2 f>c}. The proof of Theorem \ref{thm:shear} will be given in the next section.

\subsection{The Green function and Biot-Savart law}

The key to controlling the non-local effect is the exponential decay of the Green function in the infinite channel.

Throughout the paper, we denote  $z =(x,y)$ for a point in the channel with $x \in \R$, $y \in [0,1]$ being the two components. Consider the Dirichlet Green function on $\Omega = \R \times [0,1]$:
\begin{equation}\label{eq:Green}
\begin{cases}
\Delta G(z,z') = \delta(z-z') &\\
G(z,z') = 0 \quad z,z' \in \p\Omega .&  
\end{cases}
\end{equation} 
 The Dirichlet boundaries ($y=0,1$) induce a spectral gap that forces the Green's function to screen interactions exponentially. Indeed, the solution to \eqref{eq:Green} is  given explicitly by
\begin{equation}
 G(x,y, x',y') = -\frac{1}{4\pi} \ln \left( \frac{\cosh\left[{\pi}(x-x')\right] - \cos\left[{\pi}(y+y')\right]}{\cosh\left[{\pi}(x-x')\right] - \cos\left[{\pi}(y-y')\right]} \right) .
\end{equation}

The velocity $u : \Om \to \R^2$ can be recovered by the vorticity via the Biot-Savart law
\begin{equation}
    u(z):=\nabla^{\perp} \Delta^{-1}\om = \int_{\Omega } K(z,z')\om(z')\,dz'.
\end{equation}
Note that our convention of the Biot-Savart law is consistent with the counter-clockwise rotation of a positive vortex.

The following lemma demonstrates the exponential screening effect in the infinite channel. We omit the standard proof.
\begin{lemma}\label{lemma:kernal K}
    The kernel $K =(K^x, K^y) : \Omega \times \Om \to \R^2$ satisfies
    \begin{equation}\label{eq bound K}
        |K(z,z')| \lesssim 
    \begin{cases}
  \frac{1}{|z-z'|}  &\quad\text{if $|z-z'|\le 1 $  } \\
  e^{-{\pi}|x-x'|} &\quad\text{if $|z-z'| \ge 1 $.  } 
\end{cases}
    \end{equation}
    Moreover, 
    for any $\om \in L^\infty(\Om)$,
    \begin{equation}\label{eq: lemma:kernal K}
        |K*\om ( z) -  K*\om ( z')| \lesssim  \| \om\|_{L^\infty} h(|z-z'|)
    \end{equation}
    with the usual log-Lipschitz modulus $ h(r) = r (1 - \ln r)$ for $0< r\leq 1$ and $h(r) =1 $ when $r \geq 1$.
\end{lemma}

\subsection{The Yudovich theory in the infinite channel}

The existence and uniqueness of Yudovich solutions in the infinite channel $\Omega$ follow from standard arguments.

We note that the exponential decay of the Biot-Savart kernel ensures the velocity is well-defined assuming only bounded initial data, and hence the $L^1$ condition in the usual Yudovich  theory can be dropped.
 
\begin{proposition}\label{prop:Yudovich existence}
Let $f\in C^2([0,1])$. For any $\om_{in} \in L^\infty(\Om)$, there exists a unique Yudovich solution $\om \in C_w([0, +\infty ); L^\infty (\Om))$ to \eqref{eq:eu_eq} whose velocity field is log-Lipschitz as in \eqref{eq: lemma:kernal K}.

If in addition $\om_{in} \in L^p(\Om)$ for some $1\leq p < \infty$, then $\om \in C_w([0,+\infty ); L^p (\Om))$ as well.
\end{proposition}

Next, we derive useful estimates for the Yudovich solutions in the channel $\Om$. Thanks to the exponential decay of the Biot-Savart kernels, both the vorticity and the velocity field stay exponentially localized throughout the evolution.

From now on, we will assume the compact support of the initial data, though all results in this paper also hold for perturbations that are exponentially localized.

\begin{lemma}\label{lemma:apriori est}
Let $f \in  C^2([0,1])$ and $\om$  be a Yudovich solution to \eqref{eq:eu_eq} with compactly supported initial data $\om_{ in} \in L^\infty(\Om)$. Then the following estimates hold.
    \begin{itemize}
    \item For any $t \ge 0$, there hold 
    \begin{equation}\label{eq:lemma:apriori est 1}
       \| \om (t)\|_{L^{\infty}} \le    \|f''\|_{L^{\infty}}+ \| \om_{in}\|_{L^{\infty}} ,
   \end{equation}
 and
    \begin{equation}\label{eq:lemma:apriori est 2}
        \|u(t)\|_{L^{\infty}  } \lesssim   \|\om (t)\|_{L^{\infty} }.
    \end{equation}

        \item For any $0\leq |a| < {\pi}  $,  $e^{ - a x } \om   , e^{  - a x } u \in   L^1 \cap   L^\infty  $  and
        \begin{equation}\label{eq:lemma:apriori est 3}
            \| e^{   - a x } u\|_{L^p} \lesssim \big( {\pi} - |a| \big)^{-1} \|e^{   - a x} \om \|_{L^p} .
        \end{equation}

    \end{itemize}
\end{lemma}
\begin{proof}
Since the total vorticity  is transported by the flow, for any $(x,y)$ there exists $(X,Y)$ such that $\om(x,y) -f'(y) = \om_{in}(X,Y) -f'(Y)$. Then \eqref{eq:lemma:apriori est 1}  follows from  the mean value theorem:
\begin{equation}
\| \om  \|_{L^\infty} \leq    \| f'(y)-f'(Y)\|_{L^\infty} + \| \om_{in}  \|_{L^\infty}\leq    \| f''\|_{L^\infty} + \| \om_{in}  \|_{L^\infty} .
\end{equation}
The estimate \eqref{eq:lemma:apriori est 2} follows from Lemma \ref{lemma:kernal K}.

Next, we prove \eqref{eq:lemma:apriori est 3}. It suffices to consider any $0\leq  a  < {\pi}  $, and we have 
\begin{align*}
  \big| u  e^{ - a x} \big|  &  \leq \int |K (z,z')| e^{a(x' -x)}\phi(z') |\,dz' \nonumber \\ 
  &\lesssim  \int_{|z-z'| \le 1 } |z-z'|^{-1} e^{a } \phi(z')  \,dz' + \int_{|z-z'| \ge 1 }  e^{- {\pi}|x-x'| } e^{a(x' -x )}\phi(z')  \,dz'    .
 \end{align*}
So \eqref{eq:lemma:apriori est 3} follows from Young's inequality and the bound $  e^{a } \lesssim 1  \lesssim  (  {\pi}  - a  )^{-1}$. 
 
 Finally, we prove $e^{ - a x } \om   , e^{  - a x } u \in   L^1 \cap   L^\infty  $. Thanks to \eqref{eq:lemma:apriori est 3}, we only need to show  $e^{ - a x } \om \in   L^1 \cap   L^\infty  $.

Let $\phi = e^{ - a x } \om$. Since $\om_{in}$ has compact support, initially $\phi(0)\in L^1 \cap L^\infty$. We use a decoupling argument below to show $\phi(t)$ also inherits this property.

Treating $u$ as a known velocity field, the equation for $\phi$ becomes linear:
\begin{equation}\label{eq eq for we^{-ax}}
    \p_t \phi + a( f(y)+u^x)\phi +   f(y)\p_x \phi + u\cdot \nabla \phi = -u^y f''(y) e^{-ax}
\end{equation}
can be recast into the abstract form
\begin{equation}\label{eq eq for S}
\p_t\phi- \mathcal{L}  \phi = \mathcal{T}\phi,
\end{equation}
where the linear operator $\mathcal{T}:\phi \mapsto -u^y( e^{ax} \phi ) f''(y) e^{-ax} $.

Our goal is then showing the inhomogeneous equation \eqref{eq eq for S} propagates $ L^1\cap L^\infty$ regularity. To this end, let $S(t )$ be the semigroup of the homogeneous linear   equation $\p_t\phi- \mathcal{L}  \phi =0$. Since the total velocity $u+ (f(y),0)$ is bounded for any $t\ge 0$, we have 
\begin{equation}\label{eq esti for S}
\|S(t )\phi\|_{L^1  \cap L^{\infty}} \leq C(\|u\|_{L^\infty}, \|f \|_{L^\infty} )  e^{C|t |}\|\phi\|_{L^1  \cap L^{\infty}}.
\end{equation}

On the other hand, thanks to \eqref{eq:lemma:apriori est 3} the linear source term $\mathcal{T}\phi $ also obeys the bound
\begin{equation}\label{eq est for T}
\| \mathcal{T}\|_{L^p \to L^p} \lesssim \|f'' \|_{L^\infty} (  {\pi}  - |a | )^{-1} \quad \text{for any $1\leq p\leq \infty$}.
\end{equation}    

Therefore, by the Duhamel principle for $S(t )$ and the uniqueness of \eqref{eq eq for S} using \eqref{eq esti for S} and \eqref{eq est for T}, we obtain $\phi(t)  \in L^1 \cap L^\infty $ for any $t\geq 0$.

\end{proof}

\subsection{A maximum principle}

We proceed to show that the strong  exponential decay property of the Biot-Savart kernel allows us to extract the local damping effect of the transport term $f(y)\p_x \om$, leading to exponential decay on any compact sets. 

Here, we assume  a uniform lower bound $f(y)\ge c >0$  to suppress the non-local effects from both the nonlinear advection and linear stretching. The degenerate case, where $f(y)$ vanishes on the boundary, is more difficult and will be treated using energy-renormalization methods in the next section.

From now on, we fix $a = \frac{\pi}{2   } >0$ in the exponential weight and recall that $m_f  = \inf_{y \in [0,1]} f(y)>0$. We emphasize the weight  $e^{-a x}$ provides structural compatibility between the transport direction and the kernel decay, which is unique to shear flows in the infinite channel. Theorem \ref{thm:shear2 f>c} now follows from the following maximum principle.

\begin{lemma}\label{lemma:max_principle}
There exists $C_*>0$ such that the following holds. For any $f(y)$ satisfying \eqref{eq:thm 2 f >c shear assump} and a Yudovich solution  $\om$ of \eqref{eq:eu_eq} with compactly supported initial data $\om_{in} \in L^\infty$ satisfying \eqref{eq:thm 2 f >c perturb assump}, the weighted vorticity
$$
\Phi(x,y,t): = \om(x,y,t)  e^{-a (x- \frac{1}{2} m_f  t) }
$$   
satisfies the maximum principle:
\begin{equation}\label{eq:lemma:max_principle conclu}
 \| \Phi(t) \|_{L^\infty} \leq  \| \Phi(0) \|_{L^\infty}   \quad \text{for all $t\geq 0$}.
\end{equation}
\end{lemma}
\begin{proof}
Thanks to Lemma \ref{lemma:apriori est}, $\| \Phi(t) \|_{L^\infty}<\infty$ for any $t\ge 0$. We switch to Lagrangian coordinates to handle the transport term. Let $Z(x,y,t)= (Z^x, Z^y): \Om \times [0,\infty ) \to \Om$ be the unique flow of $ (f(y) , 0 ) +u $, namely
\begin{equation}
\begin{cases}
    \frac{d Z}{dt} = (f(Z^y),0 ) + u(Z)  & \\
    Z(z)|_{t=0} = z .&
\end{cases}
\end{equation}
Since the total velocity field  $ (f(y) , 0 ) +u $ is log-Lipschitz, it suffices to prove the maximum principle for 
$$
 \overline{\Phi}: = \Phi \circ Z  =  e^{\frac{1}{2}a m_f  t } \Big(
 \om  e^{-a  x} \Big) \circ Z  .
 $$  
By Lemma \ref{lemma:apriori est}, $\overline{\Phi} \in L^\infty([0,T]; L^1 \cap L^\infty)$ for any $T>0$.

The inverse flow allows us to eliminate the transport terms, and the new function  $\overline{\Phi} (x,y,t): \Om \times [0,\infty ) \to \R $  is a weak solution to the equation
\begin{equation}\label{eq: eq for  overline Phi }
\begin{cases}
    \p_t \overline{\Phi} + a(  f\circ Z -   \frac{1}{2} m_f +u^x\circ Z)\overline{\Phi}   = -e^{\frac{1}{2} a m_f  t } (u^y f''(y) e^{-ax})\circ Z  &\\
   \overline{\Phi}|_{t= 0} = \om_{in}e^{-ax} \in L^1 \cap L^\infty  . &
\end{cases}
\end{equation}
Observe from \eqref{eq: eq for  overline Phi } that $t \mapsto \overline{\Phi}  $ is Lipschitz.

Fix   $\epsilon>0$ small.  It suffices to consider any point $ (x,y,t) $ in $\Omega \times \R^+ $ such that 
\begin{equation}\label{eq: eq for  overline Phi 0}
\overline{\Phi}(x,y,t) \geq (1-\epsilon)   \|\overline{\Phi}(t) \|_{L^\infty} : =(1-\epsilon) M(t)>0 . 
\end{equation}
An identical argument will cover the negative case $\overline{\Phi}(x,y,t) \le -(1-\epsilon) M(t) $.

Then at any point $ (x,y,t) $ in $\Omega \times \R^+ $  where \eqref{eq: eq for  overline Phi 0} holds, we have
\begin{align}\label{eq: eq for  overline Phi 1}
      \p_t \overline{\Phi} & \leq  a ( -\frac{ 1}{2}m_f +  |u^x|) \overline{\Phi}   + e^{\frac{1}{2} a  m_f t}  \|f''\|_{L^\infty} \| u^y   e^{-a x}  \|_{L^\infty} .
\end{align}
We first consider the first term on the right-hand side. Lemma \ref{lemma:apriori est} implies that
\begin{equation}\label{eq: eq for  overline Phi ux}
| u^x| \lesssim    \| f''\|_{L^\infty} +    \| \om_{in}\|_{L^\infty} .
\end{equation}
Choosing the universal constant $C_*>0$ small in the two assumptions \eqref{eq:thm 2 f >c shear assump} and \eqref{eq:thm 2 f >c perturb assump}, by \eqref{eq: eq for  overline Phi ux} we can ensure that $( -\frac{ 1}{2}m_f +  |u^x|) \leq    -\frac{ 1}{4}m_f  $.  Consequently, we have obtained a negative bound for the first term in \eqref{eq: eq for  overline Phi 1}
\begin{equation}\label{eq: eq for  overline Phi 1a}
    a ( -\frac{ 1}{2}m_f +   |u^x|) \overline{\Phi} \leq - \frac{\pi}{8 }   m_f M(t).
\end{equation}
For the second term in \eqref{eq: eq for  overline Phi 1}, using \eqref{eq:lemma:apriori est 3} and \eqref{eq: eq for  overline Phi 0} we obtain 
\begin{equation}\label{eq: eq for  overline Phi 1b}
\begin{aligned}
  e^{\frac{1}{2} a  m_f t}  \|f''\|_{L^\infty} \| u^y   e^{-a x}  \|_{L^\infty} & \leq   C   e^{\frac{1}{2} a  m_f t}  \|f''\|_{L^\infty} \| \om   e^{-a x}  \|_{L^\infty} \\
  &\leq  C     \|f''\|_{L^\infty}   M (t) .
\end{aligned}
\end{equation}
Combining \eqref{eq: eq for  overline Phi 1a} and \eqref{eq: eq for  overline Phi 1b} in \eqref{eq: eq for  overline Phi 1}, we can find  a small universal constant $C_*>0$ such that under the assumptions \eqref{eq:thm 2 f >c shear assump} and \eqref{eq:thm 2 f >c perturb assump}, there holds
\begin{align*}
\p_t \overline{\Phi} & \leq - \frac{\pi}{8 }    m_f M(t)  + e^{\frac{1}{2} a  m_f t} \|f''\|_{L^\infty} \| u^y e^{-a x}  \|_{L^\infty}  \\
      & \leq \Big(- \frac{\pi m_f}{8 }     +  C   \|f''\|_{L^\infty} \Big)  M (t)  <0 .
\end{align*}

We have proven for any $ (x,y,t) $ such that $ \overline{\Phi}(x,y,t) \geq (1-\epsilon)  \sup_{x,y}  \overline{\Phi}   $ there holds $\p_t \overline{\Phi} <0$, from which we immediately have $\| \overline{\Phi}(t) \|_{L^\infty}  \leq \| \overline{\Phi} (0) \|_{L^\infty} <\infty $.
\end{proof}

With the exponential in time decay of the weighted vorticity $\om e^{-ax}$, it is standard to obtain similar decay estimates and convergence for the associated velocity. 

\begin{corollary}
Let  $\om(t)$ be a Yudovich solution in Theorem \ref{thm:shear2 f>c} and $u(t)$ be its velocity field. Then for any $R>0$,  
 \begin{equation}\label{eq:thm:shear conclusion u}
        \| u (t) \|_{ L^\infty ( (-\infty,R]\times [0,1])} \lesssim_{R} e^{- \frac{1}{2} m_f t} \qquad   t \ge 0  .
    \end{equation}  
\end{corollary}
\begin{proof}
By \eqref{eq:lemma:apriori est 3} from Lemma \ref{lemma:apriori est} and Lemma \ref{lemma:max_principle}, for $a=\frac{\pi}{2}$ we have
\begin{equation}
\| u  e^{-ax} \|_{ L^\infty  } \lesssim  \| \om  e^{-ax} \|_{ L^\infty  } \lesssim e^{- \frac{1}{2} m_f t}
\end{equation}
and hence
\begin{equation}
\| u   \|_{ L^\infty ((-\infty,R] \times [0,1] ) } \lesssim_R  e^{- \frac{1}{2} m_f t} .
\end{equation}
\end{proof}

\section{Renormalization and conserved quantities}\label{sec:renorm}

The purpose of this section is two-fold. First, we adapt the renormalization framework of DiPerna-Lions~\cite{MR1022305} to our settings, laying the groundwork for the energy argument in the next section. Second, we leverage this framework, alongside the natural conserved quantities of our system, to prove the non-compactness results of Theorem \ref{thm:noncompact}.

\subsection{Renormalization identity}

The renormalization identity \eqref{eq:prop:renormal} below  tracks the flow of information without the derivatives required for standard energy estimates; this allows us to   compute the evolution of the weighted mass \eqref{def eq:main renormal M} in the stability argument in the next section.

\begin{proposition}\label{prop:renormal}
Let $f \in C^2$ and $\om$ be a Yudovich solution to \eqref{eq:eu_eq} with compactly supported initial data $\om_{ in} \in L^\infty(\Om)$. 

Let  $\beta:\R \to \R $ be $C^1$ (possibly unbounded) with $\beta(0) = 0$  and let $\varphi \in C^\infty_c(\Omega)$. 

Then the map $t \mapsto \int \beta(\omega) \varphi \, dz$ is bounded and Lipschitz continuous, and there holds 
\begin{equation}\label{eq:prop:renormal}
\frac{d}{dt} \int_\Om \beta(\omega) \varphi \, dz = \int_\Om \beta (\omega )  v \cdot \nabla \varphi   \, dz -\int_\Om f''(y) u^y \beta'(\omega) \varphi \, dz ,
\end{equation}
where $v: = (f(y) , 0) + u$ denotes the total velocity.
\end{proposition}
\begin{proof}
Before entering the proof, let us note that in our convention $\Om = \R \times [0,1]$ is closed, so the test functions $\varphi$ in \eqref{eq:prop:renormal} are not required to vanish at $\partial \Om$.

We first prove \eqref{eq:prop:renormal} for test functions $\varphi$ supported away from the boundary, i.e. $\Supp \varphi \subset \R \times (0,1)$. 
For $\ep>0$ sufficiently small such that
$
\ep<\frac12\operatorname{dist}(\Supp \varphi,\partial\Omega),
$
we define the spatial mollification on $\Supp \varphi$ by
$$
\om_{\ep}(z,t)=\eta_{\ep} * \om (z,t) 
:=
\int_{\Omega}\eta_\varepsilon(z-z')\omega(t,z')\,dz',
\qquad z\in \text{$\Supp \varphi$} .
$$
Then for sufficiently small $\ep>0$, we have  
\begin{equation}\label{eq:aux renor om_ep}
\p_t \omega_\ep + v\cdot \nabla \omega_\ep = -(f'' u^y)_\ep + R_\ep \quad \text{on $\Supp \varphi$ }
\end{equation}
where $(\cdot )_\ep= (\cdot)* \eta_\ep $ refers to the spatial mollification and $R_\ep$ is the DiPerna-Lions commutator
\begin{equation}
R_\ep(z) = v\cdot \nabla \omega_\ep - (v\cdot \nabla \omega)_\ep  .
\end{equation}
Since $\om \in L^\infty $ and $ v \in W^{1,p}_{loc}$  with uniform in time bounds, we have the classical DiPerna-Lions commutator estimates,
\begin{equation}\label{eq:aux renormal 2}
    \| R_\ep \|_{L^1_t L^1(\Supp \varphi)}  \to 0.
\end{equation}

Now we multiply \eqref{eq:aux renor om_ep} by $\beta'(\om_\ep) \varphi$ and integrate on $\Omega \times [0,t]$.   Integrating by parts in space for the transport term, we have
\begin{equation}\label{eq:aux renormal 1}
\begin{aligned}
  & \int_\Om \beta(\omega_\ep(t)) \varphi   -    \int_\Om \beta(\omega_\ep(0)) \varphi   \\
  & =   \int_0^t \int_\Om \beta (\omega_\ep)  v \cdot \nabla \varphi    -\int_0^t \int_\Om \beta'(\omega_\ep) (f'' u^y)_\ep  \varphi  + \int_0^t \int_\Om \beta'(\omega_\ep) R_\ep  \varphi   \\
  &: = I_\ep + J_\ep + K_\ep ,
\end{aligned}
\end{equation}
where the boundary terms at $y=0,1$ vanish since $\varphi$ is supported inside. 

Since $\beta $ is $C^1(\R)$ with $\beta(0) = 0$, $\beta (\omega_\ep) -\beta (\omega ) \to 0$ in $L^1_t L^1$ so the first term $I_\ep$ converges to its natural limit as $\ep \to 0$.

Since $(f'' u^y)_\ep \to  f'' u^y  $ in $L^1_t L^1$, it suffices to show $\beta'(\omega_\ep) \to \beta'(\omega) $   a.e.  in space-time by  dominated convergence. Using that  $\om_\ep \to \om $  $a.e.$ on $\Omega$ and $\beta'$ is continuous,  we obtain that $J_\ep $ converges to the desired limit as $\ep \to 0$.

Finally, thanks to \eqref{eq:aux renormal 2} $K_\ep \to 0 $ as $\ep \to 0$.

Passing to the limit $\ep \to 0$ in \eqref{eq:aux renormal 1}, we obtain that for any $\varphi$ supported away from $\partial \Omega$,
\begin{equation}\label{eq:aux renormal 3}
\begin{aligned}
  & \frac{d}{dt }\int_\Om \beta(\omega ) \varphi \, dz  & =    \int_\Om \beta (\omega )  v \cdot \nabla \varphi   \, dz-  \int_\Om \beta'(\omega )  f'' u^y   \varphi \, dz .
\end{aligned}
\end{equation}

To conclude \eqref{eq:prop:renormal} from \eqref{eq:aux renormal 3}, we take another approximation of the test functions $\chi_\ep(y) \varphi $ where $\chi_\ep(y) $ is a vertical cutoff  function that vanishes for $d(y)\le \ep $ and $\chi_\ep(y) = 1$ for $d(y)\ge 2\ep $. Then by the boundary conditions $v^y |_{\partial \Om} =0 $, $ \beta (\omega )  v \cdot \nabla (\chi_\ep \varphi) $ converges in $L^1_t L^1 $  to $\beta (\omega )  v \cdot \nabla  \varphi  $, which implies \eqref{eq:prop:renormal} from \eqref{eq:aux renormal 3}.
\end{proof}

For application of the energy method in the next section, we record here a variation of the renormalized identity.
\begin{corollary}\label{cor:renormal}

The previous proposition holds for $\beta(s)= |s|$ and test functions $\varphi \in C^\infty$ that satisfy 
\begin{equation}
\limsup_{|z| \to \infty } \frac{|\varphi|}{e^{a|z|}} <\infty \quad \text{for some $  a <\pi $.}
\end{equation}

\end{corollary}
\begin{proof}
By taking approximations $\beta_\ep(s) = \sqrt{s^2+\ep^2} - \ep$, we see that the renormalization identity \eqref{eq:prop:renormal} holds for $\beta(s)= |s|$ and $\varphi \in C^\infty_c$.

The claim that $\varphi$ can grow exponentially follows from the  dominated convergence; indeed, the exponential integrability of   $\om$ and $u$ is given by Lemma \ref{lemma:apriori est} since the initial data $\om_{in}$ has compact support.
\end{proof}

\subsection{Conservation laws with background shears}\label{subsec:conservation}

We now examine the conservation laws associated with \eqref{eq:eu_eq}. In our infinite channel setting, the presence of a non-decaying background shear causes standard conserved quantities of the 2D Euler equations---namely, total kinetic energy and enstrophy---to become infinite.

In order to establish the non-compactness statements  in the setting of a background shear in the infinite channel, we work with \emph{excess} or \emph{relative} counterparts of those conservation laws.

Precisely, we observe that for sufficiently smooth solutions   $\om$, the \emph{excess kinetic energy}
\begin{equation}\label{eq:def_energy}
E_u(t) := \int_{\Om} |(f(y), 0)+ u|^2 - |f(y)|^2 \, dz
\end{equation}
and the \emph{excess enstrophy}
\begin{equation}\label{eq:def_enstrophy}
E_{\om}(t) : =\int_{\Om} |-f'(y)+ \om |^2 - |f'(y)|^2 \, dz
\end{equation}
are finite and conserved quantities. Furthermore, the entire family of \emph{excess Casimirs}
\begin{equation}\label{eq:def_Casimir}
C_{\om}(t) : =\int_{\Om} \Phi(-f'(y)+ \om ) - \Phi( -f'(y) ) \, dz
\end{equation}
is conserved by \eqref{eq:eu_eq}. See~\cite{Mcintyre_Shepherd_1987}. These conservation laws serve as natural analogues to Arnold's energy-Casimir framework in  bounded domains~\cite{arnold66}.

We now pass to the renormalization framework to establish these claims.
\begin{proposition}\label{prop:conseved}
Let $f \in C^2([0,1])$ and let $\om (t)$ be a Yudovich solution to \eqref{eq:eu_eq} with compactly supported initial data $\om_{in} \in L^\infty$.

For any $\Phi \in C^1(\R)$, the excess Casimir of the solution $\om(t)$, defined by  \eqref{eq:def_Casimir}, is constant. 

Similarly, the  excess kinetic energy  is also constant.

\end{proposition}
\begin{proof}
Consider the total vorticity $W = -f'(y)+ \om \in L^\infty$ and the total velocity $v = (f(y) , 0) + u$. Observe that  $W$ satisfies the transport equation in the sense of distribution (compactly supported test functions)
\begin{equation}\label{eq: aux prop:conseved 1}
    \p_t W + v \cdot \nabla W = 0.
\end{equation}
Due to the regularity $W \in L^\infty$ and $v \in  W^{1,p}_{loc}$ for all $p<\infty$, $W$ is a renormalized solution in the sense of DiPerna-Lions~\cite{MR1022305}: for any $\Phi \in C^1(\R)$, 
\begin{equation}\label{eq: aux prop:conseved 2}
   \p_t \Phi (W)  + v \cdot \nabla \Phi (W) = 0 \quad \text{in the sense of distributions}.
\end{equation}
Note that test functions in $C^\infty_c(\Omega)$ for both \eqref{eq: aux prop:conseved 1} and \eqref{eq: aux prop:conseved 2} can take non-zero boundary values because  the channel definition $\Om = \R \times [0,1]$ imposes no-penetration boundary conditions.

We are interested in the excess density $\eta(z,t): = \Phi(W) - \Phi(-f'(y))$. Since $\Phi \in C^1$ and $f \in C^2$, this quantity $\eta : \Om \times [0,\infty) \to \R$  satisfies
\begin{equation}
\p_t \eta  + v \cdot \nabla \eta  =  u^y f''(y) \Phi'(-f'(y) )   
\end{equation}
in the sense of distribution. Moreover, the bound $ |\eta| \leq |\Phi'|_{L^\infty} |\om|  $ implies $\eta \in L^\infty( [0,T]; L^1\cap L^\infty)$ for any $T>0$.

Now to find the rate of change of the total excess Casimir $C_\omega(t) = \int_\Omega \eta dz$, we introduce for any $R>0$  a smooth horizontal cutoff function   $\varphi_R(x)$ such that $\varphi_R(x) = 1$ for $|x| \le R$ and $\varphi_R(x) = 0$ for $|x| \ge 2R $. It follows that
\begin{equation}\label{eq:aux prop:conseved 1}
\frac{d}{dt} \int_\Omega \eta \varphi_R dz =  \int_\Omega \eta  (v  \cdot \nabla  ) \varphi_R dz + \int_\Omega u^y f''(y) \Phi'(-f'(y)) \varphi_R dz .
\end{equation}
By the $L^1$ integrability of $\eta$ and $u^y$, we see that as $R \to \infty$,
\begin{equation}\label{eq:aux prop:conseved 2}
\begin{aligned}
 \int_\Omega \eta  (v  \cdot \nabla  ) \varphi_R dz & \to 0  \\
\int_\Omega u^y f''(y) \Phi'(-f'(y)) \varphi_R dz & \to  \int_\Omega u^y f''(y) \Phi'(-f'(y))   dz.
\end{aligned}
\end{equation}
It follows from \eqref{eq:aux prop:conseved 1} and \eqref{eq:aux prop:conseved 2} that for any $t\geq 0$, there holds
\begin{equation}
  \int_\Omega \eta(z,t)   dz  -   \int_\Omega \eta(z,0)   dz  =   \int_0^t \int_\Omega u^y f''(y) \Phi'(-f'(y))   dz dt .
\end{equation}
The boundary condition $u^y =0$ allows us to integrate by parts to see
\begin{equation}
    \int_\Omega u^y f''(y) \Phi'(-f'(y))   dz =     \int_\Omega \p_y u^y   \Phi(-f'(y)) = \int_\Omega -\p_x u^x   \Phi(-f'(y)) =0 .
\end{equation}
Hence $t\mapsto \int_\Omega \eta(z,t)   dz$ is constant.

The verification of \eqref{eq:def_energy} follows from the velocity formulation of the total velocity $v = (f(y) , 0 ) + u$. Since $v \in W^{1,p}_{loc}$, all the computations can be carried out at the level of $L^p$. We omit the computations.

\end{proof}

\subsection{Loss of compactness}\label{subsec:compactness}

Finally, we demonstrate that weakly convergent Yudovich solutions exhibit a loss of compactness arising from the transport of bulk vorticity to spatial infinity.

We can now prove Theorem \ref{thm:noncompact}  thanks to Proposition \ref{prop:conseved}.
\begin{proof}[Proof of Theorem \ref{thm:noncompact}]

We divide the proof into two cases.

\vspace{0.5em}
\noindent
\textbf{Case 1: the linear flows $f''(y) = 0$}

In this case, the right-hand side of \eqref{eq:eu_eq} vanishes, and the vorticity is transported by the total velocity $ v  =(f(y) , 0 ) + u $. Thus, $\| \om(t)\|_{L^p} = \| \om_{in}\|_{L^p} $ implies that the trajectory $\om(t)$ cannot be relatively compact in $L^p$ for any $1 \le p \le \infty$.

Now consider the trajectory $t \mapsto u(t)$ under the assumption that the initial excess energy $ E_{u} \ne 0 $. If there exists $t_n \to \infty$ such that $ u (t_n) $ is strongly convergent in $L^1$, then the limit has to be zero by the weak convergence $\om(t)\rightharpoonup 0$ assumption. Moreover, by Lemma \ref{lemma:apriori est}, there holds the bound $\| u (t_n)\|_{L^2}^2 \lesssim \| u (t_n)\|_{L^1} \| u (t_n)\|_{L^\infty} $, which implies strong convergence $ u (t_n) \to 0$ in $L^2$. As a result, we see  that
\begin{equation}\label{eq: aux thm:noncompactness 1}
\begin{aligned}
|E_u (t_n) | & \leq  2  \Big| \int f(y) u  dz \Big|  + \| u (t_n) \|_{L^2}^2  \\
& \leq  \| f\|_{L^\infty}\| u(t_n)\|_{L^1} + \| u(t_n)\|_{L^2}^2 \to 0
\end{aligned}
\end{equation}
a contradiction to $E_u(t) = E_{u}(0) \ne 0$.

\vspace{0.5em}
\noindent
\textbf{Case 2: general shear flows}

The general idea is similar to the previous case. Since the solution $\om(t)$ preserves these non-zero conserved quantities, strong convergence to zero is impossible unless the conserved quantities vanish identically.

Let us first consider the case where the initial excess enstrophy  $E_{\om}(0)  \ne 0$.

Assume that there exists $t_n \to \infty$ and $ \om(t_n) \to 0 $ strongly in $L^1$. Then by Lemma \ref{lemma:apriori est} again, $ \om(t_n) \to 0 $ strongly in $L^2$. Therefore, 
\begin{equation}\label{eq: aux thm:noncompactness 2}
\begin{aligned}
|E_\om (t_n) | & \leq  2  \Big| \int f'(y) \om  dz \Big|  + \| \om (t_n) \|_{L^2}^2  \\
& \leq  \| f'\|_{L^\infty}\| \om (t_n)\|_{L^1} + \| \om (t_n)\|_{L^2}^2 \to 0
\end{aligned}
\end{equation}
contradicting $E_{\om}(0)  \ne 0$.

Finally, we consider the case where the initial excess energy  $E_{u}(0)  \ne 0$. Then the same argument for \eqref{eq: aux thm:noncompactness 1} shows that $t \mapsto u(t)$ is not relatively compact in $L^1(\Omega)$. However, if $ \om(t_n ) \to 0$ strongly in $L^1$, then Lemma \ref{lemma:apriori est} also shows $ u(t_n ) \to 0$ strongly in $L^1$. So this case is also impossible.

\end{proof}

\section{Stability for degenerate shear   flows}\label{sec:degenerate}

In this section, we present the proof of Theorem \ref{thm:shear}, establishing that shear flows satisfying \eqref{eq:thm:shear assump f 1} and \eqref{eq:thm:shear assump f 2} are asymptotically stable. Our approach builds upon the ideas developed in the previous sections. However, unlike the pointwise decay arguments that rely on maximum principles, we adopt an energy-renormalization strategy based on a weighted $L^1$ mass functional.

The primary analytical challenge lies in the absence of a uniform lower bound on $f(y)$ near the boundary, where stagnation could potentially trap vorticity. Our energy-renormalization argument demonstrates that, even in this degenerate setting, the escape mechanism prevails. Our rationale is twofold:
\begin{itemize}
\item Linear stretching: We rely on the curvature bound \eqref{eq:thm:shear assump f 2} and the  boundary condition $u^y = 0$. Intuitively, condition \eqref{eq:thm:shear assump f 2} ensures that the generation of disturbances via linear stretching occurs at a slower rate than advection by the background shear.

\item Nonlinear advection: We utilize the symmetry enforced by our $L^1$ mass functional. Heuristically, the back-flow region is confined to a thin boundary layer, thereby inducing a vanishing weight for the vorticity within the double integral.
\end{itemize}

\subsection{Monotonicity of the weighted mass}

We are in a position to prove Theorem \ref{thm:shear}. The energy functional we use will penalize the movement of vorticity to the left, using the same exponential weight as in the previous section. This is possible thanks to Lemma \ref{lemma:apriori est}.

Now consider a shear profile $ f(y)$ satisfying the criterion \eqref{eq:thm:shear assump f 1} and \eqref{eq:thm:shear assump f 2}. The threshold constant $C_*>0$ will be taken to be small in the end, depending on a few universal constants.

Recall that $a =\frac{\pi}{2}$. Let us define the weighted mass
\begin{equation}\label{def eq:main renormal M}
    M(t) := \int_{\Om} |\om (z,t)|e^{-ax}\,dz.
\end{equation}
Our energy argument will focus on the evolution of $M$. Although our solutions only have $\om \in L^\infty$ regularity and the background shear is only $C^2$, thanks to Proposition \ref{prop:renormal} and Corollary \ref{cor:renormal} we have
\begin{equation}\label{eq:main renormal}
    \begin{aligned}
        \frac{dM}{dt} 
        &= -\frac{\pi}{2}\int f(y)|\om |e^{-ax}\,dz-\frac{\pi}{2}\int u^x |\om| e^{-ax} \,dz \\
        & \qquad -\int \sign(\om )  f''(y)u^y e^{-ax}\,dz.
    \end{aligned}
\end{equation}

In the rest of this section, we will first prove the following.
\begin{proposition}\label{prop:energy prop}
There exists $C_* >0$ such that if the shear profile $f \in C^2([0,1])$ satisfies \eqref{eq:thm:shear assump f 1} and \eqref{eq:thm:shear assump f 2} and the compactly supported initial data $\om_{in}\in L^\infty$ satisfies \eqref{eq:thm:shear perturb assump} with a sufficiently small $\ep =\ep(f)$, then
\begin{equation}
    \begin{aligned}
        \frac{dM}{dt} 
        &\leq -\frac{\pi{\delta}}{4}\int_{\Om} d(y) |\om  (z)|e^{-ax}\,dz .
    \end{aligned}
\end{equation}
\end{proposition}

\subsection{Estimates for the nonlinear part}
 
We now establish the estimate for the contribution of the nonlinear term to the weighted mass evolution \eqref{eq:main renormal}.

We recall that the horizontal velocity field $u^x: \Omega \times [0,\infty) \to \mathbb{R}$ is recovered via the Biot-Savart law $u^x(z)=\int K^x(z,z')\omega(z')\,dz'$, where the kernel decomposes into two components $K^x = K_1^x + K_2^x$: 
\begin{align}
K_1^x  &: = \frac{1}{4  }    \frac{-  \sin(\pi(y-y'))}{\cosh(\pi(x-x'))-\cos(\pi(y-y'))} \label{eq:def_K^x_1}\\
 K_2^x  &: =   \frac{1}{4  } \frac{\sin(\pi(y+y'))}{\cosh(\pi(x-x'))-\cos(\pi(y+y'))}   \label{eq:def_K^x_2}.
\end{align}

The key lemma below establishes the estimate for the nonlinear term, exploiting certain double symmetry in its $L^1$ formulation.

\begin{lemma}\label{lemma:u^x energy}
    Let $d(y) = \min\{y, 1- y \}$. For any $\om   \in L^{\infty}(\Omega)$, there holds
    \begin{equation}
         \int |u^x(z)| |\om  (z)| e^{-ax} \,dz \lesssim \| \om  \|_{L^{\infty}}   \int d(y)  | \om  | e^{-ax}  . 
    \end{equation} 
\end{lemma}
\begin{proof}
For brevity, let us define the non-negative integrand
\begin{equation}\label{eq: aux lemma:u^x energy 1}
F_i(z,z'):=\left|K_i^x(z,z')\right| |\om(z')\om(z)|e^{-ax}.
\end{equation}
By expanding $u^x$ using \eqref{eq:def_K^x_1} and \eqref{eq:def_K^x_2}, it is enough to prove, for $i=1,2$, that
\begin{equation}\label{eq: aux lemma:u^x energy 2}
\iint_{\Omega\times\Omega} F_i(z,z')\,dz'\,dz
\lesssim
\|\om\|_{L^\infty}
\int_\Omega d(y)|\om(z)|e^{-ax}\,dz.
\end{equation}

We decompose the square
$[0,1]\times[0,1]$ in the variables $(y,y')$ into four regions:
\begin{equation}\label{eq: four regions y yp}
\begin{aligned}
A_1&:=\{(y,y'):\ y\le y',\ y+y'\le 1\},\\
A_2&:=\{(y,y'):\ y'\le y,\ y+y'\le 1\},\\
A_3&:=\{(y,y'):\ y\le y',\ y+y'\ge 1\},\\
A_4&:=\{(y,y'):\ y'\le y,\ y+y'\ge 1\},
\end{aligned}
\end{equation}
and consider the corresponding decomposition for $\iint  F_i(z,z')\,dz'\,dz$:
\begin{equation}\label{eq: decomposed four regions y}
\iint_{\Omega\times\Omega}F_i(z,z')\,dz'\,dz
=
  \iint_{A_1 \cup A_4}F_i(z,z')\,dz'\,dz +   \iint_{A_2 \cup A_3}F_i(z,z')\,dz'\,dz,
\end{equation}
where  the two terms have their integration over $(y,y')$ restricted to $A_1 \cup A_4$ and respectively $A_2 \cup A_3$.

For later use, let us first record the elementary $x$-integral bounds. Using \eqref{eq:def_K^x_1} and \eqref{eq:def_K^x_2}, for $K_1^x$, 
\begin{equation}\label{eq:K1x bound}
|K_1^x(z,z')|
\lesssim
\begin{cases}
\displaystyle
\frac{|y-y'|}{|x-x'|^2+|y-y'|^2},
& |z-z'|\le 1,\\[1em]
\displaystyle
|y-y'| e^{-\pi|x-x'|},
& |z-z'|\ge 1.
\end{cases}
\end{equation}
and for $K_2^x$, 
\begin{equation}\label{eq:K2x bound}
|K_2^x(z,z')|
\lesssim
\begin{cases}
\displaystyle
\frac{|y+y'|}{|x-x'|^2+|y+y'|^2},
& |z-z'|\le 1,\\[1em]
\displaystyle
|y+y'| e^{-\pi|x-x'|},
& |z-z'|\ge 1.
\end{cases}
\end{equation}
Since $a=\pi/2<\pi$, \eqref{eq:K1x bound} and \eqref{eq:K2x bound} imply the uniform estimates
\begin{equation}\label{eq: uniform x integral bounds}
\int |K_i^x(z,z')|e^{-a(x-x')}\,dx \lesssim 1,
\qquad
\int |K_i^x(z,z')|\,dx' \lesssim 1.
\end{equation}
We demonstrate the first bound in \eqref{eq: uniform x integral bounds} for $K_1^x$, as the estimates for the second one are the same. Indeed, for $K_1^x$, in the near-field region $|z-z'|\le 1$, we have $|x-x'|\le 1$, hence the weight
$e^{-a(x-x')}$ is bounded by a constant, and
$$
\int_{|x-x'|\le 1}
\frac{|y-y'|}{|x-x'|^2+|y-y'|^2}\,dx
\lesssim
\int_{\mathbb R}
\frac{|y-y'|}{ x^2+|y-y'|^2}\,dx
\lesssim 1.
$$
In the far-field region, using $a<\pi$ and $|y-y'| \leq 1$, 
$$
\int |y-y'| e^{-\pi|x-x'|}e^{-a(x-x')}\,dx
\lesssim
|y-y'|\int_{\mathbb R}e^{-(\pi-a)|x|}\,dx
\lesssim 1.
$$
To show the bound for $K_2^x$ in \eqref{eq: uniform x integral bounds}, one replaces $|y-y'|$ in the above argument by $|y+y'|$.

With \eqref{eq: uniform x integral bounds} established for $K^x_1$ and $K^x_2$, we now estimate the two pieces in the decomposition \eqref{eq: decomposed four regions y} separately.

\vspace{0.5em}
\noindent
\textbf{Case 1: Estimates in $A_1\cup A_4$.}

In the region $A_1$, for fixed
$y'\in[0,1]$, the admissible values of $y$ are given by
\begin{equation}\label{eq:y range A_1}
\begin{aligned}
\{y:\ (y,y')\in A_1\}
 & = \{y:\ 0\le y\le y',\ y+y'\le 1\}\\
 & = [0,\min\{y',1-y'\}] = [0,d(y')],
\end{aligned}
\end{equation}
and similarly, in the region $A_4$, we have
\begin{equation}\label{eq:y range A_4}
\begin{aligned}
\{y:\ (y,y')\in A_4\}
 & = \{y:\ 0\le y'\le y \le 1,\ y+y'\ge 1\}\\
 & = [ \max\{y',1-y'\}, 1] = [1- d(y') , 1].
\end{aligned}
\end{equation}
Therefore both of these regions have $y$-projection of size $d(y') $.

Using $e^{-ax}=e^{-ax'}e^{-a(x-x')}$, bounding $|\om(z)|$ by
$\|\om\|_{L^\infty}$, and applying \eqref{eq: uniform x integral bounds}, \eqref{eq:y range A_1} and \eqref{eq:y range A_4}, we obtain
\begin{equation}\label{eq: A1 estimate}
\begin{aligned}
& \iint_{A_1 \cup A_4 }F_i(z,z')\,dz'\,dz \\
&\lesssim
\|\om\|_{L^\infty}
\int_\Omega |\om(z')|e^{-ax'} \times
\left(
\int_{ \{ y: (y,y')\in A_1 \cup A_4\} }
\int |K_i^x(z,z')|e^{-a(x-x')}\,dx\,dy
\right)
dz'
\\
&\lesssim
\|\om\|_{L^\infty}
\int_\Omega d(y')|\om(z')|e^{-ax'}\,dz'.
\end{aligned}
\end{equation}

\vspace{0.5em}
\noindent
\textbf{Case 2: Estimates in $A_2\cup A_3$.}

 In the region $A_2$, for fixed $y\in[0,1]$, the admissible values of $y'$ are
\begin{equation}\label{eq:y range A_2}
\begin{aligned}
\{y':\ (y,y')\in A_2\}
 & = \{y':\ 0\le y'\le y,\ y+y'\le 1\}\\
 & = [0,\min\{y,1-y\}] = [0,d(y)],
\end{aligned}
\end{equation}
while in the region $A_3$, we have
 \begin{equation}\label{eq:y range A_3}
\begin{aligned}
\{y':\ (y,y')\in A_3\}
 & = \{y':\ y\le y'\le 1,\ y+y'\ge 1\} \\
 & = [\max\{y,1-y\},1] = [1-d(y),1].
\end{aligned}
\end{equation}

Now we will integrate in $x',y'$ first. Bounding $|\om(z')|$ by $\|\om\|_{L^\infty}$ and applying
\eqref{eq: uniform x integral bounds}, \eqref{eq:y range A_2}, and \eqref{eq:y range A_3} , we get
\begin{equation}\label{eq: A2 estimate}
\begin{aligned}
& \iint_{A_2 \cup A_3  }F_i(z,z')\,dz'\,dz  \\
& \lesssim
\|\om\|_{L^\infty}
\int_\Omega |\om(z)|e^{-ax} \times
\left(
\int_{y' : (y,y') \in A_2\cup A_3}
\int |K_i^x(z,z')|\,dx'\,dy'
\right)
dz
\\
&\lesssim
\|\om\|_{L^\infty}
\int_\Omega d(y)|\om(z)|e^{-ax}\,dz.
\end{aligned}
\end{equation}

Combining \eqref{eq: A1 estimate} and  
\eqref{eq: A2 estimate}, we obtain
$$
\iint_{\Omega\times\Omega}F_i(z,z')\,dz'\,dz
\lesssim
\|\om\|_{L^\infty}
\int_\Omega d(y)|\om(z)|e^{-ax}\,dz.
$$
\end{proof}

\subsection{Estimates for the linear stretching}
Next, we bound the contribution of the linear source term $-f''(y)u^y$. It is for this specific term that we must impose an upper bound on the curvature $|f''|$ near the boundaries. Note that compared to the cases of periodic channels~\cite{MR4628607,MR4740211},  we do not need $f''$ to vanish near the boundary.

Let us recall the formula for $K^y(z,z')$:
\begin{equation}\label{eq:def_K^y}
\begin{aligned} 
K^y(z,z') & = \frac{1}{2 } \frac{\sinh(\pi(x-x')) \sin(\pi y) \sin(\pi y')}{\left[ \cosh(\pi(x-x')) - \cos(\pi(y-y')) \right] \left[ \cosh(\pi(x-x')) - \cos(\pi(y+y')) \right]} .
\end{aligned}
\end{equation}

\begin{lemma}\label{lemma:u^y energy}
If $|f''(y)| \le c_f d(y)$ for some $c_f > 0$, then  for any $ \om  \in L^{\infty}$,   there holds
\begin{equation} \label{eq:lemma:u^y energy}
        \int \left| f''(y)u^ye^{-ax}\right| \,dz  \lesssim   c_f \int  d(y) | \om |e^{-ax}\,dz. 
\end{equation}
\end{lemma}
\begin{proof}
We begin by substituting the Biot-Savart formula $u^y(z) = \int K^y(z,z') \omega(z') \, dz'$ into the left-hand side. Using the curvature hypothesis $|f''(y)| \le c_f d(y)$ and applying Fubini's theorem to interchange the order of integration, we obtain 
\begin{equation}\label{eq bound K^y0} 
    \begin{aligned}
        \int \left| f''(y)u^ye^{-ax}\right| \,dz &\lesssim c_f \iint     d(y) e^{-a(x-x')}|K^y(z,z')| \,dz \,| \om  (z')|e^{-ax'}\,dz'.
    \end{aligned}
\end{equation}

By comparing \eqref{eq bound K^y0} with our target estimate \eqref{eq:lemma:u^y energy}, the lemma reduces to proving the following uniform bound on the inner integral
\begin{equation}\label{eq aux bound K^y} 
\int_{ \Omega }d(y) e^{-a(x-x')} |K^y(z,z')| \, dz   \lesssim d(y') \quad \text{for all $z'\in \Om$}.
\end{equation}

To establish \eqref{eq aux bound K^y}, we consider separately the two cases $y' \le \frac12$ and $y' \ge \frac12$.

\vspace{0.5em}
\noindent
\textbf{Case 1: Estimates in the region $y' \le \frac12$.} 

In this region, the distance to the boundary is simply $d(y') = y'$. 
Similarly to the bounds derived for $K^x$ in Lemma \ref{lemma:u^x energy},  we can bound the weighted kernel pointwise as follows:
\begin{equation}\label{eq bound K^y1} 
        d(y)|K^y(z,z')| \lesssim
    \begin{cases}
  \frac{|x -x'| y^2 y' }{(|x-x'|^2+|y-y'|^2) (|x-x'|^2+|y+y'|^2) }  &\quad\text{if $|z-z'|\le 1 $  } \\
 y^2 y' e^{-\pi|x-x'|} &\quad\text{if $|z-z'| \ge 1 $.  } 
\end{cases}
\end{equation}
We can simplify the near-field bound by estimating the non-singular denominator terms from below, yielding the simplified upper bounds:
\begin{equation}\label{eq bound K^y2} 
        d(y)|K^y(z,z')| \lesssim
    \begin{cases}
  \frac{y' }{|z-z'| }  &\quad\text{if $|z-z'|\le 1 $  } \\
 y^2 y' e^{-\pi|x-x'|} &\quad\text{if $|z-z'| \ge 1 $. }
\end{cases}
\end{equation}

We now integrate these pointwise bounds against the exponential weight $e^{-a(x-x')}$. For the near-field region ($|z-z'| \le 1$), the singularity $1/|z-z'|$ is   integrable in 2D, giving the estimates
\begin{equation}
\begin{aligned}
\int_{|z-z'| \le 1 } d(y) e^{-a(x-x')}  |K^y(z,z')| \, dz  & \lesssim \int \frac{y'}{|z-z'|}    \lesssim y'.
\end{aligned}
\end{equation}

For the far-field region ($|z-z'| \ge 1$), the exponential decay of the kernel $e^{-\pi|x-x'|}$ dominates the weight $e^{-a(x-x')}$ (since $a=\frac{\pi}{2} < \pi$). Thus,  we also obtain the same desired bound:
\begin{equation}
\int_{|z-z'| \ge 1 } d(y) e^{-a(x-x')}  |K^y(z,z')| \, dz \lesssim y' \int_{|z-z'| \ge 1} y^2 e^{-(\pi-a)|x-x'|} \,dz \lesssim y'.
\end{equation}

Combining these two cases, we conclude that for any $ y' \le \frac12 $ 
\begin{equation}\label{eq bound K^y2 b} 
\int_{ \Omega }e^{-a(x-x')} d(y)|K^y(z,z')| \, dz   \lesssim y'   = d(y').
\end{equation}

\vspace{0.5em}
\noindent
\textbf{Case 2:  Estimates in the region $y' \ge \frac12$.} 

By the symmetry of the channel, the analysis for $y' \ge \frac12$ (where $d(y') = 1-y'$) is entirely analogous.

We adapt the pointwise bounds to reflect distances from the top boundary 
\begin{equation}\label{eq bound K^y4} 
        d(y)|K^y(z,z')| \lesssim
    \begin{cases}
  \frac{1-y' }{|z-z'| }  &\quad\text{if $|z-z'|\le 1 $  } \\
 (1-y)^2 (1-y') e^{-\pi|x-x'|} &\quad\text{if $|z-z'| \ge 1 $.  }
\end{cases}
\end{equation}
Integrating these bounds exactly as in Case 1 yields the corresponding upper bound:
\begin{equation}\label{eq bound K^y5}
\int_{ \Omega } e^{-a(x-x')} d(y)|K^y(z,z')| \, dz   \lesssim 1-y' = d(y').
\end{equation}

\vspace{0.5em}
\noindent
\textbf{Conclusion.}

We have   established \eqref{eq aux bound K^y} uniformly across the entire domain $z'\in \Om$ by \eqref{eq bound K^y2 b} and \eqref{eq bound K^y5}. Substituting it back into \eqref{eq bound K^y0}   concludes the proof:
    \begin{align*}
        \int \left| f''(y)u^ye^{-ax}\right| \,dz &\lesssim c_f \int_{y' \le \frac12}    y'   \,| \om  (z')|e^{-ax'}\,dz'+ c_f \int_{y' \ge \frac12}   ( 1-y' )  \,| \om  (z')|e^{-ax'}\,dz' \\
        &\lesssim c_f \int    d(y')   \,| \om  (z')|e^{-ax'}\,dz' .
    \end{align*}

 \end{proof}

\subsection{Conclusion of the proof of Proposition \ref{prop:energy prop}}
 
With Lemma \ref{lemma:u^x energy} and Lemma \ref{lemma:u^y energy} we can conclude Proposition \ref{prop:energy prop}.

\begin{proof}[Proof of Proposition \ref{prop:energy prop}]

Since for some $\delta>0$,   $|f''(y)| \le C_* \delta d(y)$, using the estimates from Lemma \ref{lemma:u^x energy} and Lemma \ref{lemma:u^y energy}, we have
\begin{equation}
\begin{aligned}
       \frac{dM}{dt} 
        &  \leq -\frac{\pi}{2}\int f(y)|\om |e^{-ax}\,dz  \\
        & \qquad + C \delta C_* \int d(y)|\om |e^{-ax}\,dz + C \| \om \|_{L^\infty} \int d(y)|\om |e^{-ax}\,dz,
\end{aligned}
\end{equation}
where $C>0$ denotes some universal constants independent of $f$ or $\om$. Note that  we have not yet chosen the constant $C_*$ nor the constraint $\ep(f)$ for the initial perturbation.

Since $f(y) \ge \delta d(y)$, there hold
\begin{equation}\label{eq:aux proof of prop 3.3 1}
\begin{aligned}
       \frac{dM}{dt} 
        &  \leq -\frac{\pi \delta}{2}\int d(y)|\om |e^{-ax}\,dz  \\
        & \qquad + C \delta C_* \int d(y)|\om |e^{-ax}\,dz + C \| \om \|_{L^\infty} \int d(y)|\om |e^{-ax}\,dz.
\end{aligned}
\end{equation}  

Now we can first choose $C_*>0$ in the shear flow assumption \eqref{eq:thm:shear assump f 2} so that the second term on the right-hand side of \eqref{eq:aux proof of prop 3.3 1} can be absorbed by the first negative damping term:
\begin{equation}\label{eq:aux proof of prop 3.3 2}
\begin{aligned}
       \frac{dM}{dt} 
        &  \leq -\frac{3\pi \delta}{8}\int d(y)|\om |e^{-ax}\,dz    + C \| \om \|_{L^\infty} \int d(y)|\om |e^{-ax}\,dz.
\end{aligned}
\end{equation}

Similarly, we can further choose the perturbation threshold $\ep(f)$ in \eqref{eq:thm:shear perturb assump} (depending on $\delta>0$ and $C_*>0$) so that the remaining positive term on the right-hand side of \eqref{eq:aux proof of prop 3.3 2} can also be absorbed, yielding the desired conclusion for any shear flow $f$ satisfying \eqref{eq:thm:shear assump f 1} and \eqref{eq:thm:shear assump f 2}:
\begin{equation}\label{eq:aux proof of prop 3.3 3}
\begin{aligned}
       \frac{dM}{dt} 
        &  \leq -\frac{ \pi \delta}{4}\int d(y)|\om |e^{-ax}\,dz    .
\end{aligned}
\end{equation}

\end{proof}
 
Now Proposition \ref{prop:energy prop} states that the weighted $L^1$ mass is non-increasing in time, and we need to upgrade this to the local vanishing statement in Theorem \ref{thm:shear} to conclude. 

The lemma below achieves this final goal, based on the idea that the bulk of the vorticity cannot concentrate on where the weight $d(y)$ vanishes.
\begin{lemma}\label{lemma:final}
If the solution $\om(t)$ satisfies
\begin{equation}\label{eq:lemma:final}
    \frac{dM}{dt} \le -\frac{\pi\delta}{4}\int_{\Om} d(y) | \om  (z)|e^{-ax}\,dz,
\end{equation}
then for any $N >0 $, there holds
    \begin{equation}\label{eq:lemma:final conclu}
       \lim_{t\to +\infty}  A_N(t):=\lim_{t\to +\infty} \int_{ |x| \le N, \, d(y) \ge \frac1N} | \om  (z)|\,dz =0.
    \end{equation}
\end{lemma}
\begin{proof}
    Assuming \eqref{eq:lemma:final conclu} does not hold,  there exists a fixed $N>0$ and a sequence of times $t_n \to \infty$ such that
    \begin{equation}\label{aux  lemma:final conclu 1}
        A_N(t_n) \ge \frac{1}{N} >0. 
    \end{equation}
Since $M(t)$ is non-increasing, in this case we also must have
    \begin{equation}\label{aux  lemma:final conclu 2}
        \lim_{t\to +\infty}M(t)=M_{\infty}>0.
    \end{equation}

\textbf{Claim:} There exists $T =T(N, \om_{in}, f)>0$ independent of $n$ such that for any $t \in [t_n, t_n +T ]$
    \begin{equation}\label{aux  lemma:final conclu 3}
    \begin{aligned}
    A_{2N}(t) &   \ge   \frac{1}{2}A_{N}(t_n )   =\frac{1}{2N } .
    \end{aligned} 
    \end{equation}
To prove this claim, we observe that thanks to the renormalization property in Corollary \ref{cor:renormal}, $|\om|$ satisfies 
\begin{equation}\label{aux  lemma:final conclu 4}
    \p_t |\om | + v \cdot \nabla  |\om| = -f''(y)u^y \sign(\om)
\end{equation}
where $v =(f(y),0) +u$ denotes the total velocity. 

Consider for $t\ge t_n$ the flow map $Z(t)$ of $v$ starting at $t=t_n$, namely
\begin{equation}
    \begin{cases}
        \frac{d Z(z,t)}{dt} = v(Z(z,t),t) &\\
        Z(z,t_n) = z \quad \text{for all $z\in  \Om$}.&
    \end{cases}
\end{equation}
It follows from \eqref{aux  lemma:final conclu 4} that upon defining   $|\om | \circ Z = |\om( Z(z,t),t)|$, there holds
\begin{equation}\label{aux  lemma:final conclu 5}
\begin{cases}
   \frac{d |\om | \circ Z  }{dt}  =  \left( -f''(y)u^y \sign(\om) \right)\circ Z  \,  &  \\
    |\om | \circ Z |_{t =t_n } = |\om|(t_n) .&
\end{cases}
\end{equation}
Since by Lemma \ref{lemma:apriori est} the total velocity $\|u+(f(y),0)\|_{L^{\infty}} \lesssim 1$ for all times, we first take $T>0$ small such that $ |Z(z,t) -z  |_{L^\infty}  \leq (2N)^{-1}$. This implies that
\begin{equation}\label{aux  lemma:final conclu 6}
   \int_{|x| \le 2N, d(y)\ge (2N)^{-1}} |\om(z,t)| \, dz \ge  \int_{|x| \le  N, d(y)\ge N^{-1}} |\om(Z(z,t),t)| \, dz .
\end{equation}
On the other hand, integrating \eqref{aux  lemma:final conclu 5} on $[t_n,t]$ gives
\begin{equation}\label{aux  lemma:final conclu 7}
    |\om(Z(z,t),t)| \ge |\om(z,t_n)| - \int_{t_n}^t    | f''(y) u^y(z,\tau) | \circ Z(z,\tau )\, d\tau .
\end{equation}
Further integrating \eqref{aux  lemma:final conclu 7} on the spatial domain $|x| \le N , d(y) \ge N^{-1}$ gives
\begin{equation}\label{aux  lemma:final conclu 8}
\begin{aligned}
\int_{|x| \le  N, d(y)\ge N^{-1}} |\om(Z(z,t),t)| \, dz &    \ge  \int_{|x| \le  N, d(y)\ge N^{-1}} |\om(z,t_n)| \, dz \\
& \qquad - 4 (t-t_n)\| f''\|_{L^\infty} \| u\|_{L^\infty} N^2 ,
\end{aligned}
\end{equation}
where we have bounded the source term by its supremum times the total measure of the domain.

By choosing $T>0$ smaller (depending on $f$ and $\om_{in}$, thanks to Lemma \ref{lemma:apriori est}) if needed, it follows from \eqref{aux  lemma:final conclu 6} and \eqref{aux  lemma:final conclu 8}  that
\begin{equation}\label{aux  lemma:final conclu 9}
   \int_{|x| \le 2N, d(y)\ge (2N)^{-1}} |\om(z,t)| \, dz \ge  \int_{|x| \le  N, d(y)\ge N^{-1}} |\om(Z(z,t),t)| \, dz \ge (2N)^{-1}.
\end{equation}
Thus, the claim has been proven.

Finally, from the claim, we see that for any $t \in [t_n,t_n +T ]$, there holds
    \begin{equation}\label{aux  lemma:final conclu 10}
    \begin{aligned}
    \int_{\Om} d(y) | \om  (z,t)|e^{-ax} \,dz  
     & \ge  \int_{x \le 2N, d(y)   \ge \frac{1}{2N}} d(y) | \om  (z,t)|e^{-ax} \,dz\\
    & \ge \frac{e^{-2aN}}{2N} \int_{x \le 2N, d(y) \ge \frac{1}{2N}}  | \om  (z,t)|  \,dz \\
    & \ge \frac{e^{-2aN}}{2N} A_{2N}(t)  \ge  \frac{  e^{-2aN}}{2N^2}  : = C_N > 0.
    \end{aligned}
    \end{equation}
    By this lower bound, \eqref{eq:lemma:final}   implies that along the sequence $t_n$, we have
    \begin{equation}\label{aux eq:lemma:final conclu}
        M(t_n+T) \le M(t_n)- T C_{ N }. 
    \end{equation}
    Since $t_n \to \infty$, we can iterate \eqref{aux eq:lemma:final conclu} infinitely many times, which leads to a contradiction to \eqref{aux  lemma:final conclu 2}. This completes the proof.
\end{proof}

By upgrading the local decay to be uniform in $y$, we can establish the vanishing statement \eqref{eq:thm:shear conclusion} in Theorem \ref{thm:shear}. 
 
\begin{corollary}\label{cor:vanishing om}
    For any $N >0 $, there holds
    \begin{equation}\label{eq:co:final conclu}
       \lim_{t\to +\infty}   \int_{ |x| \le N} | \om  (z)|\,dz =0.
    \end{equation}
\end{corollary}
\begin{proof}
    For any $\delta>0$, we show that there exists $T_{\delta}>0$ such that
    $$\int_{ |x| \le N} | \om  (z,t)|\,dz \le \delta$$ 
    for any $t \ge T_{\delta}$. Choosing $M>N$ such that $\frac{2N\| \om(t)\|_{L^{\infty}}}{M} \le \frac{\delta}{2}$, then we have
    \begin{equation*}
        \begin{aligned}
            \int_{ |x| \le N} | \om  (z,t)|\,dz &= \int_{ |x| \le N,d(y) \le \frac1M} | \om  (z,t)|\,dz+\int_{ |x| \le N,d(y)\ge \frac1M} | \om  (z,t)|\,dz \\
            & \le \frac{4N\| \om (t)\|_{L^{\infty}}}{M}+A_{M}(t).
        \end{aligned}
    \end{equation*}
    Thanks to \eqref{eq:lemma:final conclu}, taking $T_{\delta}$ large enough such that $A_M(t)\le \frac{\delta}{2}$ for any $t \ge T_{\delta}$, we obtain the desired conclusion.
\end{proof}
 
Finally, we note that, as in the previous section, the local convergence of the vorticity implies that of the velocity field. 

\begin{corollary}\label{cor: u decay}
Let  $\om(t)$ be a Yudovich solution in Theorem \ref{thm:shear} and $u(t)$ be its velocity field.   For any $R>0$, we have 
    $$\lim_{t \to \infty} \|u(t)\|_{L^{\infty}((-R, R]\times[0,1])}=0.$$
\end{corollary}
\begin{proof}
   Fix $R>0$. We need to prove that for any  $\delta>0$, there exists $T>0$ such that for any $t \ge T$,
    \begin{equation}
        \|u(t)\|_{L^{\infty}((-R, R]\times[0,1])} \le \delta . 
    \end{equation}
For a sufficiently large $N$, consider the decomposition
\begin{equation}\label{aux cor: u decay 1}
\begin{aligned}
\|u \|_{L^{\infty}((-R, R]\times[0,1])} & \leq   \big\|u[ \om   \mathbf{1}_{|x|\le N }] \big\|_{L^{\infty}}  \\
& \quad     + \sup_{|x| \leq R} |u[\om   \mathbf{1}_{|x|\ge N} ] | .
\end{aligned}
\end{equation}
    
We first choose $N \gg R$  such that   the exponential decay of the Biot-Savart kernel implies  
\begin{equation}\label{aux cor: u decay 3}
\sup_{|x| \leq R}\big|u[ \om \mathbf{1}_{|x| \ge N }] (z,t) \big| \lesssim \| \om (t)\|_{L^{\infty}}e^{-\pi(N-R)}\le  \frac{\delta}{2}.
\end{equation}
Note \eqref{aux cor: u decay 3} holds uniformly in time.

Next, by Corollary \ref{cor:vanishing om}, we have
\begin{equation}\label{aux cor: u decay 4}
\int_{|x| \le N  } |\om(z,t)| \, dz  \to 0 \quad \text{as $t \to \infty$},
\end{equation}
and hence by choosing $T$ sufficiently large it follows that
\begin{equation}\label{aux cor: u decay 5}
\big\|u[ \om   \mathbf{1}_{|x|\le N }] \big\|_{L^{\infty}} \leq \frac{\delta}{2} \quad \text{for all $t \ge T$}.
\end{equation}

Finally, combining   \eqref{aux cor: u decay 3} and \eqref{aux cor: u decay 5} with \eqref{aux cor: u decay 1} gives
\begin{equation*}
        \|u(t)\|_{L^{\infty}((-R, R]\times[0,1])} \le \delta \quad \text{for all $t \ge T$}. 
    \end{equation*}    
    
\end{proof}

\section{Discussion}\label{sec:discuss}

Shear flows are exact stationary solutions across a vast hierarchy of domains, ranging from bounded periodic channels to the unbounded whole plane. Our results highlight that hydrodynamic stability is not intrinsic to the flow profile alone; rather, it is fundamentally determined by the interplay between the domain geometry and the class of admissible perturbations.

The Couette flow provides a striking illustration of this dichotomy. Although the base shear profile remains identical in both settings, the $x$-periodicity imposed in previous studies~\cite{BedrossianMasmoudi15,MR4076093} fundamentally alters the long-time dynamics. In the infinite channel $\mathbb{R} \times [0,1]$, the open topology allows vorticity to escape to spatial infinity.

This ``escape mechanism'' precludes the recirculation and echoes that drive instability in periodic domains~\cite{MR4630602}, thereby permitting asymptotic stability even for rough (Yudovich) perturbations. To our knowledge, this result represents one of the rare instances in the analysis of nonlinear PDEs where the regularity required for asymptotic stability coincides exactly with the threshold for well-posedness.

Finally, we note that extending this asymptotic stability to general shear flows $f>0$ in the whole plane $\mathbb{R}^2$ or the half plane remains open. In the  finite-width channel, the Dirichlet boundaries induce a spectral gap that enforces exponential decay of the Green’s function. In contrast, the Green's function in $\mathbb{R}^2$ or $\mathbb{R}\times \R^+$ decays only algebraically, leading to stronger interactions that will likely require new techniques to control.

\bibliographystyle{plain}
\bibliography{euler_stability}

\end{document}